\documentclass[12pt]{amsart}

\usepackage{amsmath, amssymb}

\usepackage[dvipdfmx]{graphics}
\usepackage{here}
\usepackage{tikz}
\usetikzlibrary{arrows.meta}

\usepackage{amsmath, amssymb}
\usepackage{pifont}
\usepackage{booktabs}

\numberwithin{equation}{section}
\newtheorem{theorem}{Theorem}[section]  
\newtheorem{theorem?}{``Theorem''}[section]

\newtheorem{lemma}[theorem]{Lemma}

\theoremstyle{definition}

\theoremstyle{remark}
\newtheorem{remark}[theorem]{Remark}

\newcommand{\R}{{\mathbb R}}
\newcommand{\C}{{\mathbb C}}
\newcommand{\PP}{{\mathbb P}}

\newcommand{\N}{{\mathbb N}}

\makeindex             

\begin{document}

\title[Unbounded Reinhardt domains]{
Unbounded Reinhardt domains with finite-dimensional Bergman spaces in $\C^n$}
\author{Chika Hayashida and Joe Kamimoto}

\address{Kyushu university, 
Motooka 744, Nishi-ku, 
Fukuoka, 819-0395, Japan.}  
\email{hayashida.chika.311@s.kyushu-u.ac.jp}

\address{Kyushu university, 
Motooka 744, Nishi-ku, 
Fukuoka, 819-0395, Japan.}  
\email{joe@math.kyushu-u.ac.jp}
%
%
\maketitle

\begin{abstract}
In this paper, we construct unbounded domains 
in $\C^n$ ($n\geq 2$), 
whose Bergman spaces are nontrivial and 
finite-dimensional. 
We further show that 
the Bergman metrics on these domains 
have positive constant sectional curvature 
equal to $2$, and 
that  their holomorphic automorphism groups 
consist only of linear mappings. 
\end{abstract}

\section{Introduction}

In the study of several complex variables, 
the distinction between whether a domain is 
bounded or unbounded has, 
at times, 
not only caused serious difficulties but also 
raised interesting issues
(see 
\cite{Wie84,Zwo99, Zwo00,CKO04,Eng07,Juc12,KNY14,AGK16,
PfZ17, Yam17, GHH17, HST17, HST20, HST21, HLT25}, etc.).
In this note, we construct specific unbounded domains 
in 
$\C^n$ ($n\geq 2$), 
which appear to be of interest and significant 
from the viewpoints of the following three aspects:
the dimensions of the Bergman spaces, 
the holomorphic sectional curvatures of the Bergman metrics, 
and the holomorphic automorphism groups. 
These three aspects are, in fact,  
closely interconnected. 

Our example of such domains is defined as follows. 
Let $a>0$ and 
consider the unbounded complete Reinhardt domain 
$D(a) \subset \mathbb{C}^n$ ($n \geq 2$) given by
\begin{equation}\label{eqn:1.1}
D(a)=
\left\{
z\in\C^n: 
\left||z_k|^2-|z_n|^2\right|<\frac{1}{(|z_n|^2+1)^{a}}
\text{ for $k=1,\ldots,n-1$}
\right\}.
\end{equation}
The construction of the above $D(a)$ 
is inspired by the following two known 
examples in $\C^2$:
\begin{equation}\label{eqn:1.2}
\begin{split}
&D_{\rm Wie}(k):=
\left\{z\in\C^2:  
\big||z_1|-|z_2|\big|\cdot (|z_1|+|z_2|)^{4k}<1
\right\};\\
&D_{\rm HL}(\alpha):=
\left\{z\in\C^2: 
\left||z_1|^2-|z_2|^2\right|<\frac{1}{(1+|z_1|^2+|z_2|^2)^{\alpha}}
\right\},
\end{split}
\end{equation}
where $k\in\N$ 
and $\alpha\in(2,3)$.
The above two examples 
are 
given in the papers by Wiegerinck \cite{Wie84} and
by Huang, Li and Treuer \cite{HLT25}, respectively. 
(The domain $D_{\rm Wie}(k)$ is modified from our perspective.) 
It is worth noting that the previously known examples 
are two-dimensional, 
whereas our example is constructed in arbitrary dimension $n$. 
While these three domains are similar in shape, 
$D(a)$ is slightly different from $D_{\rm Wie}(k)$ 
and $D_{\rm HL}(\alpha)$, 
even in the two-dimensional case.
We remark that 
the boundary of $D(a)$ is not smooth 
when $n\geq 3$, but it is easy to 
construct domains with smooth boundaries
which possess the same properties 
as those of $D(a)$ obtained  
in this note (see Remark~2.2 (i) below).

This note is organized as follows. 
In Section~2, 
the Bergman space of the domain 
$D(a)$ is shown to be finite-dimensional.
In Section~3, we show that 
the Bergman metric of $D(a)$ 
for certain values of $a$
has constant 
holomorphic sectional curvature 
equal to $2$. 
In Section~4, the structure of  
the holomorphic 
automorphism group of $D(a)$ 
for certain values of $a$ is investigated.

\smallskip

{\it Notation and symbols.}\quad
\begin{itemize}
\item
We denote by $\N$, $\R$, $\C$ the set consisting of 
all natural numbers, real numbers, 
complex numbers, respectively. 
Moreover, we denote by $\N_0$
the set consisting of all nonnegative integers.
\item
For $p:=(p_1,\ldots,p_n)\in\N_0^n$, 
$z=(z_1,\ldots,z_n)\in\C^n$, 
the multi-indeces will be used as 
$
|p|:=p_1+\cdots+p_n$ and
$z^{p}=z_1^{p_1}\cdots z_n^{p_n}.
$
\end{itemize}

\section{Dimensions of the Bergman spaces}

Let $\Omega$ be a domain in $\C^n$ and 
let $A^2(\Omega)$ be the Hilbert space 
of the $L^2$-holomorphic functions on $\Omega$, i.e., 
$$
A^2(\Omega)=
\left\{
f:\Omega\to\C \mbox{ holomorphic}:
\|f\|_{\Omega}:=
\left(
\int_{\Omega} |f(z)|^2 dV(z)
\right)^{1/2}<\infty
\right\},
$$
where $dV$ is the Lebsgue measure on $\C^n$.

When $\Omega$ is a bounded domain, 
$A^2(\Omega)$ contains all the monomials, 
which implies that 
the dimension of $A^2(\Omega)$ is infinity. 
On the other hand,  
the issue of the dimension of 
$A^2(\Omega)$ is not obvious in the unbounded case. 
When $\Omega$ is a domain in the complex plane $\C$, 
it is seen in \cite{Wie84} that 
\begin{equation}\label{eqn:2.1}
\text{$A^2(\Omega)$
is either trivial, i.e., $\{0\}$,  or infinite dimensional.}
\end{equation}
It might be natural to expect that 
property (\ref{eqn:2.1}) always holds
in the general dimensional case. 
However, J. Wiegerinck \cite{Wie84} 
gave, surprisingly, a counterexample to 
(\ref{eqn:2.1}); more precisely,  he showed that
$D_{\rm Wie}(k)$ with $k\in\N$ 
has a finite-dimensional 
and nontrivial Bergman space.
More recently, 
Huang, Li and Treuer \cite{HLT25} also showed that 
the domain $D_{\rm HL}(\alpha)$ has the same properties.

We now show that the Bergman space of 
$D(a)$ is also finite-dimensional.

\begin{theorem}
If $a>\frac{1}{n-1}$, then 
the Bergman space of $D(a)$ is given by
\begin{equation}\label{eqn:2.2}
A^2(D(a))={\rm Span}\{z^p:
p\in\N_0^n, 
|p|<(n-1)a-1\}.
\end{equation}
From (\ref{eqn:2.2}), the dimension of $A^2(D(a))$ 
can be expressed as 
\begin{equation}
\dim A^2(D(a))=
\begin{cases} 
&0 \quad  \mbox{ if  $ 0<a\leq \frac{1}{n-1}$,} \\
&\binom{n+k}{k} \quad
\mbox{if  $\frac{k+1}{n-1}<a\leq 
\frac{k+2}{n-1}$ 
 $(k\in\N_0)$}.
\end{cases}
\end{equation}
\end{theorem}

\begin{proof}
Using the polar coordinates 
$z_j=r_j e^{i\theta_j}$ for $j=1,\ldots,n$
in the computation of the integral:
\begin{equation*}
\begin{split}
\|z^p\|^2=&\int_{D(a)} 
|z_1|^{2p_1} \cdots |z_n|^{2p_n}dV(z), 
\end{split}
\end{equation*}
we have 
\begin{equation*}
\begin{split}
\|z^p\|^2=
(2\pi)^n \int_0^{\infty} 
\left(
\prod_{k=1}^{n-1} 
\int_{|r_k^2-r_n^2|<\frac{1}{(r_n^2+1)^{\alpha}},r_k>0} 
r_k^{2p_k+1}dr_k
\right)
r_n^{2p_n+1} dr_n. 
\end{split}
\end{equation*}
Furthermore, 
changing integral variables, 
we have  
\begin{equation}\label{eqn:2.4}
\begin{split}
\|z^p\|^2=\pi^n \int_0^{\infty} 
\prod_{k=1}^{n-1} f_{p_k}(\rho)
\rho^{p_n} d\rho, 
\end{split}
\end{equation}
with
\begin{equation*}
f_q(\rho) = 
\int_{I(\rho)} 
r^{q}dr, 
\end{equation*}
where $q\in\N_0$ and 
$
I(\rho)=\{r>0: |r-\rho|<(\rho+1)^{-a}\}.
$
The convergence of the integral in (\ref{eqn:2.4})
is determined by the behavior of 
$f_q(\rho)$ as $\rho\to\infty$.
Its asymptotic behavior can be described as follows.
\begin{equation*} 
\begin{split}
f_q(\rho)&=
\frac{1}{q+1}
\left[
\left(
\rho+\frac{1}{(\rho+1)^{a}}
\right)^{q+1}
-
\left(
\rho-\frac{1}{(\rho+1)^{a}}
\right)^{q+1}
\right] \\
&=\frac{\rho^{q+1}}{q+1}
\left[
\left(1+\frac{1}{\rho(\rho+1)^{a}}\right)^{q+1}- 
\left(1-\frac{1}{\rho(\rho+1)^{a}}\right)^{q+1}
\right] \\
&=\frac{\rho^{q+1}}{q+1}
\left[
\left\{
1+\frac{q+1}{\rho^{a+1}}
+O\left(
\frac{1}{\rho^{a+2}}
\right)
\right\}- 
\left\{
1-\frac{q+1}{\rho^{a+1}}
+O\left(
\frac{1}{\rho^{a+2}}
\right)
\right\}
\right] \\
&=2\rho^{q-a}
\left(1+O\left(
\frac{1}{\rho}
\right)
\right) \quad \mbox{as $\rho\to\infty$.}
\end{split}
\end{equation*}
Applying the above result to (\ref{eqn:2.4}), 
we can see the following equivalences. 
\begin{equation}\label{eqn:2.5}
\begin{split}
\|z^{p}\|<\infty 
\Longleftrightarrow & \;\;
\int_1^{\infty} \rho^{|p|-a(n-1)}
d\rho<\infty  \\
\Longleftrightarrow & \;\;
|p|-a(n-1)<-1,  
\end{split}
\end{equation}
which implies 
$A^2(D(a))\supset 
{\rm Span}\{z^p:
p\in\N_0^n, 
|p|<(n-1)a-1\}$.
%
Conversely, the opposite inclusion can be 
established in a similar manner 
as in the proof of the main theorem in \cite{Wie84}.

We remark that 
since $|p|$ is nonnegative, 
the value of $a$ must be greater than 
$\frac{1}{n-1}$ from (\ref{eqn:2.5}).

\end{proof}

\begin{remark}
(i) \;\;
The boundary of $D(a)$ is not smooth when $n\geq 3$. 
However,  we can construct a domain 
with $C^{\infty}$-smooth boundary satisfies 
the properties in Theorem 2.1 as follows. 
For $a>0$, we define
\begin{equation}
\tilde{D}(a)=
\left\{
z\in\C^n: 
\left||z_k|^2-|z_n|^2\right|<\frac{2}{(|z_n|^2+1)^{a}}
\mbox{ for $k=1,\ldots,n-1$}
\right\}.
\end{equation}
Since $D(a)\subsetneq \tilde{D}(a)$
and the boundaries of $D(a)$ and $\tilde{D}(a)$ 
do not intersect, 
it is easy to see the existence of 
a Reinhardt domain $\Omega(a)$
with smooth boundary such that 
$D(a)\subsetneq \Omega(a) \subsetneq \tilde{D}(a)$.
It is easy to check that $\Omega(a)$ satisfies 
the same properties as in Theorem~2.1.

(ii)\;\;
The pseudoconvexity of a domain 
has a significant influence on 
the dimension of its Bergman space. 
Wiegerinck \cite{Wie84} conjectured that 
property (\ref{eqn:2.1}) always holds 
for every pseudoconvex domain $\Omega$.
As far as we know, this conjecture remains 
unsolved at present, 
but it has already been affirmatively 
settled in many cases
(\cite{Zwo99, Zwo00, Eng07, Juc12, 
GHH17, PfZ17}).
In particular, 
it has been confirmed 
in the case of Reinhardt domains
(\cite{Zwo99, Zwo00, Eng07}). 
From these results,  the domain
$D(a)$ with $a\geq\frac{1}{n-1}$
cannot be pseudoconvex. 
\end{remark}

\section{Holomorphic sectional curvatures 
of the Bergman metrics}

Let $\Omega$ be a domain in $\C^n$. 
The {\it Bergman kernel} $K_{\Omega}(z,w)$ of $\Omega$ 
is defined by 
$$
K_{\Omega}(z,w)=\sum_{\alpha} 
\phi_{\alpha}(z) \overline{\phi_{\alpha}(w)},
$$ 
where $\{\phi_{\alpha}\}_{\alpha}$ is a complete orthonormal 
basis of $A^2(\Omega)$.  
Its restriction to the diagonal, $K_{\Omega}(z,z)$,  
is denoted simply by $K_{\Omega}(z)$. 

It has recently been recognized 
that the issue of the dimension of $A^2(\Omega)$ is
closely related to that of 
constant holomorphic sectional curvatures
of the Bergman metric. 
Recently, 
Huang, Li and Treuer \cite{HLT25} has obtained 
several significant results on this topic.
In their paper, they show that the Bergman metric 
of the domain $D_{\rm HL}(\alpha)$ 
in (\ref{eqn:1.2}) has constant holomorphic 
sectional curvature equal to $2$. 
In this note, we present 
a general-dimensional example $D(a)$  in (\ref{eqn:1.1}) 
as follows. 

\begin{theorem}
If 
$\frac{2}{n-1}<a\leq\frac{3}{n-1}$,
then the following statements hold: 
\begin{enumerate}
\item
The Bergman space is 
$A^2(D(a))={\rm Span}\{1,z_1,\ldots,z_n\}.$
\item
The Bergman metric of $D(a)$ has constant 
holomorphic sectional curvature equal to $2$.
\end{enumerate}
\end{theorem}

\begin{proof}

(i) \;\; It follows from Theorem~2.1 that 
the condition 
$1<(n-1)a-1\leq 2$ is equivalent to (i). 

(ii) \;\;
From (i), 
the Bergman kernel of $D(a)$ can be written
as 
\begin{equation}\label{eqn:3.1}
K_{D(a)}(z)=c_0+\sum_{k=1}^n c_k |z_k|^2
\quad (z\in D(a)), 
\end{equation}
where 
$$
c_0={\rm vol} (D(a))^{-1}, 
\quad 
c_k=\|z_k\|^{-2}=
\left(\int_{D(a)} |z_k|^2 dV(z)\right)^{-1}.$$ 
Furthermore, we have
$$
K_{D(a)}(z)=
c_0\left(1+\sum_{k=1}^n \tilde{c}_k |z_k|^2\right)
\quad (z\in D(a)),
$$ 
where $\tilde{c}_k=c_k/c_0$ for
$k=1,\ldots,n$. 
Through
the scaling $(w_1,\ldots,w_n)=
(\sqrt{\tilde{c}_1}z_1,\ldots,
\sqrt{\tilde{c}_n}z_n)$,
the Bergman metric of $D(a)$ 
is isometric to the Fubini-Study metric
of $\PP^n$ restricted to 
a domain in the cell $\C^n$ with 
$$
\omega_{FS} =
\partial\bar{\partial} \log 
\left(
1+\sum_{k=1}^n |w_k|^2
\right). 
$$
Therefore, the Bergman metric of $D(a)$ 
has a constant holomorphic 
sectional curvature equal 
to $2$.
\end{proof}


\section{Holomorphic automorphism groups}

By applying the results of A. Yamamori in 
\cite{Yam14, Yam15, Yam17}
(see also \cite{KNY14}), 
we can describe the structure of the holomorphic 
automorphism group of $D(a)$ for
$a\in(\frac{2}{n-1},\frac{3}{n-1}]$. 
The results stated in this section follow directly
from a combination of Yamamori's works.
In fact, 
the argument presented below is based on an idea of 
Yamamori, 
communicated during a lecture at a symposium 
in Fukuoka (\cite{Yam25lecture}), 
which has appeared in the preprint \cite{Yam25}. 
We include it here for the reader's convenience.

Let $\Omega$ be a domain in $\C^n$ and
let ${\rm Aut}(\Omega)$ 
be the holomorphic automorphism group 
on $\Omega$. 
The $n \times n$ matrix $T_{\Omega}(z,w)$ 
is defined by
\begin{equation*}
T_{\Omega}(z,w)=
\begin{pmatrix} \displaystyle
\frac{\partial^2}{\partial z_1 \partial \overline{w_1}} 
\log K_{\Omega}(z,w) & \cdots &  
\displaystyle
\frac{\partial^2}{\partial z_n \partial \overline{w_1}} 
\log K_{\Omega}(z,w)\\
\vdots & \ddots & \vdots \\
\displaystyle
\frac{\partial^2}{\partial z_1 \partial \overline{w_n}} 
\log K_{\Omega}(z,w) & \cdots &  
\displaystyle
\frac{\partial^2}{\partial z_n \partial \overline{w_n}} 
\log K_{\Omega}(z,w)
\end{pmatrix},
\end{equation*}
for $(z,w)\in\Omega\times\Omega$ such that 
$K_{\Omega}(z,w)\neq 0$.
The function $B_{\Omega}(z)$ is defined by 
\begin{equation}\label{eqn:4.1}
B_{\Omega}(z):=
\frac{\det T_{\Omega}(z,z)}{K_{\Omega}(z)},
\end{equation}
for $z\in\Omega$ such that $K_{\Omega}(z)>0$. 
Let $\Omega_1$, $\Omega_2$ be domains
in $\C^n$ and let $F:\Omega_1\to\Omega_2$
be a biholomorphism.
From the transformation laws:
\begin{equation*}
\begin{split}
&K_{\Omega_1}(z,w)=
\overline{\det J_F(w)}\cdot
K_{\Omega_2}(F(z),F(w))\cdot
\det J_F(z), \\
&T_{\Omega_1}(z,w)=
\overline{J_F(w)^{\intercal}}\cdot
T_{\Omega_2}(F(z),F(w))\cdot
J_F(z), 
\end{split}
\end{equation*}
where $J_F(z)$ is the Jacobian matrix of $F$ at $z$, 
it follows that the function in (\ref{eqn:4.1}) is invariant 
under the biholomorphism in the sense that 
\begin{equation}\label{eqn:4.2}
B_{\Omega_1}(z)=B_{\Omega_2}(F(z)).
\end{equation}
Of course, the above transformation laws 
hold on the set where the corresponding quantities 
are well defined.

The following theorem 
is a strong generalization of 
the famous theorem of H. Cartan. 

\begin{theorem}[\cite{IsK10, Yam14,  KNY14}]
Let $\Omega_1$ 
and $\Omega_2$ be domains in $\C^n$containing the origin, 
and 
suppose that each satisfies the following two conditions:
\begin{enumerate}
\item 
$K_{\Omega}(z,0)\equiv 
K_{\Omega}(0,0)>0$ for $z\in\Omega$;
\item
$T_{\Omega}(z,0)\equiv T_{\Omega}(0,0)$ is positive definite
for $z\in\Omega$.
\end{enumerate}
Then every biholomorphism $F:\Omega_1\to\Omega_2$ 
satisfying $F(0)=0$ is linear.
\end{theorem}


Many important sufficient conditions for the 
above conditions (i) and (ii) 
have been established
(\cite{IsK10, Yam14, KNY14, Yam15, Yam17}).
In particular, 
the following lemma will be useful for our analysis.

\begin{lemma}[\cite{Yam17}]
If $\Omega$ is a Reinhardt domain 
containing the origin and 
$A^2(\Omega)$ contains $1, z_1, \ldots, z_n$, 
then $\Omega$ satisfies the conditions (i) and (ii) 
in Theorem~4.1. 
\end{lemma}

Using Theorem~4.1 with Lemma~4.2, 
we investigate the structure of 
the holomorphic automorphism group 
of the domain $D(a)$ in (\ref{eqn:1.1}). 
Since $A^2(D(a))$ contains a nontrivial
constant function by Theorem~3.1, 
it follows that  $K_{D(a)}(z)> 0$ for $z\in D(a)$, 
which ensures that the function 
$B_{D(a)}(z)$ is well-defined on $D(a)$.

\begin{theorem}
If $\frac{2}{n-1}<a \leq \frac{3}{n-1}$, 
then every $\varphi\in{\rm Aut}(D(a))$ is linear. 
\end{theorem}

\begin{proof}
Since $A^2(D(a))$ contains $\{1,z_1,\ldots,z_n\}$ 
from Theorem~2.1, Lemma~4.2 implies that 
the domain $D(a)$ satisfies
the conditions (i) and (ii) in Theorem~4.1. 
We now apply Theorem~4.1 to the case
$\Omega_1=\Omega_2=D(a)$.
It remains to prove
that every $F\in{\rm Aut}(D(a))$
satisfies $F(0)=0$. 

We prove this by contradiction. 
Let us assume that there exist 
$F_0\in {\rm Aut}(D(a))$ and 
$z_0 \in D(a) \setminus\{0\}$ such that 
$F_0 (0)=z_0$. 
From the explicit formula for the Bergman kernel  
in (\ref{eqn:3.1}), 
we can compute the exact expression of 
$B_{D(a)}$ as follows:
\begin{equation}\label{eqn:4.3}
B_{D(a)}(z)=
\frac{\prod_{j=0}^n c_j}
{K_{D(a)}(z)^{n+2}}=
\frac{\prod_{j=0}^n c_j}
{(c_0+\sum_{j=1}^n c_j |z_j|^2)^{n+2}}. 
\end{equation}
In particular, 
\begin{equation}\label{eqn:4.4}
B_{D(a)}(z_0)<B_{D(a)}(0)
\left(=\frac{\prod_{j=1}^n c_j}{c_0^{n+1}}\right).
\end{equation}
Let $w_0=F_0^{-1}(0)\;(\neq 0)$.
Then (\ref{eqn:4.2}) is equivalent to 
$B_{D(a)}(F_0(0))<B_{D(a)}(F_0(w_0))$.
From (\ref{eqn:4.2}), we have
$B_{D(a)}(0)<B_{D(a)}(w_0)$,
which is a contradiction. 
\end{proof}

\begin{theorem}
If $\frac{2}{n-1}<a_1, a_2 \leq \frac{3}{n-1}$, 
then every biholomorphism 
from $D(a_1)$ to $D(a_2)$ is linear. 
In particular, if $D(a_1)$ and $D(a_2)$ are
biholomorphically equivalent, then they
are linearly equivalent.
\end{theorem}

\begin{proof}
In a manner similar to the proof of Theorem~4.3, 
it is straightforward to see that  
every biholomorphism $F: D(a_1)\to D(a_2)$ 
satisfies $F(0)=0$. 
Therefore, from Theorem~4.1, 
every biholomorphism 
$F: D(a_1) \to D(a_2)$ is linear. 
\end{proof}

\begin{remark}
By Theorem~4.2 in \cite{Yam25}, 
the statement of Theorem 4.1 in the two-dimensional 
case can be strengthened as follows:
$D(a_1)$ and $D(a_2)$ are biholomorphically 
equivalent if and only if $a_1=a_2$.
Since
the result of Soldatkin \cite{Sol02},  
which is valid only in the two-dimensional case, 
plays a crucial role in the proof of 
\cite{Yam25}, 
it seems to be difficult to obtain
an analogous result 
in the general $n$-dimensional case. 
\end{remark}

\smallskip


{\it Acknowledgement.} \;
After completing a preliminary version of this note, 
we learned from Atsushi Yamamori that many results in Section~4 
had already appeared in his preprint~\cite{Yam25}.
We are grateful to him for bringing this to our attention.

\end{document}